\documentclass{birkjour}
\usepackage{latexsym, amsfonts, amsmath, amsthm, amssymb, verbatim, enumerate}
\usepackage{tikz}

 \newtheorem{thm}{Theorem}[section]
 \newtheorem{cor}[thm]{Corollary}
 \newtheorem{lem}[thm]{Lemma}
 \newtheorem{prop}[thm]{Proposition}
 \theoremstyle{definition}
 \newtheorem{defn}[thm]{Definition}
 \theoremstyle{remark}
 
 \newtheorem{ex}[thm]{Example}
 \numberwithin{equation}{section}
 
 \makeatletter
\@namedef{subjclassname@2010}{%
  \textup{2010} Mathematics Subject Classification}
\makeatother

\newcommand{\AC}{AC}
\newcommand{\BV}{{BV}}

\newcommand{\CTPP}{\hbox{$CT\kern-0.2ex{P}\kern-0.2ex{P}$}}
\newcommand{\Pol}{\mathcal{P}}

\newcommand{\mR}{\mathbb{R}}
\newcommand{\mC}{\mathbb{C}}

\newcommand{\mZ}{\mathbb{Z}}

\newcommand{\abs}[1]{\left\lvert#1\right\rvert}

\newcommand{\vecx}{{\boldsymbol{x}}}
\newcommand{\vecy}{{\boldsymbol{y}}}

\newcommand{\vecz}{{\boldsymbol{z}}}

\newcommand{\norm}[1]{\left\lVert#1\right\rVert}

\newcommand{\normbv}[1]{\left\lVert#1\right\rVert_{\BV(\sigma)}}

\newcommand\st{\thinspace : \thinspace}
\def\ls[#1,#2]{\overline{\vphantom{\vbox to 1.2 ex{}} #1\, #2}}
\def\lso[#1,#2]{\overline{\vphantom{\vbox to 1.2 ex{}} #1\, #2}^\circ}

\def\sparen(#1){\Bigl ( #1 \Bigr )}
\def\ssparen(#1){ (#1) }
\newcommand\plist[1]{\bigl[ #1 \bigr]}

\newcommand{\interior}[1]{\mathop{\mathrm{int}}(#1)}

\DeclareMathOperator{\var}{var}
\DeclareMathOperator{\cvar}{\rm cvar}
\DeclareMathOperator{\vf}{vf}

\DeclareMathOperator{\supp}{supp}

\newcommand{\journalname}[1]{\textrm{#1}}

\begin{document}

\title{Isomorphisms of $AC(\sigma)$ spaces for countable sets}
\author{Ian Doust and Shaymaa Al-shakarchi}

\address{School of Mathematics and Statistics,
University of New South Wales,
UNSW Sydney 2052 Australia}%
\email{i.doust@unsw.edu.au}

\date{28 November 2013}

\begin{abstract}
It is known that the classical Banach--Stone theorem does not extend to the class of $AC(\sigma)$ spaces of absolutely continuous functions defined on compact subsets of the complex plane. On the other hand, if $\sigma$ is restricted to the set of compact polygons, then all the corresponding $AC(\sigma)$ spaces are isomorphic (as algebras). In this paper we examine the case where $\sigma$ is the spectrum of a compact operator, and show that in this case one can obtain an infinite family of homeomorphic sets for which the corresponding function spaces are not isomorphic. 
\end{abstract}

\maketitle

\section{Introduction}

Well-bounded operators are one generalization of self-adjoint operators to the Banach space setting.  A bounded linear operator $T$ on a Banach space is said to be well-bounded if there is a compact interval $[a,b] \subseteq \mR$ such that $T$ admits a bounded $AC[a,b]$ functional calculus. At least on reflexive spaces such operators possess a type of spectral decomposition theory similar to that for self-adjoint operators, but one which allows conditionally rather than unconditionally convergent spectral expansions.

Even on a general Banach space, every compact well-bounded operator admits a diagonal representation of the form
  \begin{equation}\label{diag-form}
     T = \sum_{j=1}^\infty \lambda_j P_j
  \end{equation}
where $\{\lambda_j\}$ is the set of nonzero eigenvalues of $T$ and $\{P_j\}$ is the corresponding set of Riesz projections onto the eigenspaces. Conversely, under suitable conditions on $\{\lambda_j\}$ and $\{P_j\}$, any operator formed in this way is compact and well-bounded \cite{CD}. For example, if $\{Q_j\}_{j=0}^\infty$ are the projections associated with a Schauder decomposition of the space $X$ and $P_j = Q_j - Q_{j-1}$, ($j = 1,2,\dots)$ then $\sum_{j=1}^\infty \lambda_j P_j$ is compact and well-bounded for any decreasing sequence $\{\lambda_j\}$ of positive reals converging to zero. 

An obstruction to extending this theory to provide an analogue of normal operators was the lack of a good replacement for the function algebra $AC[a,b]$. Ideally the functional calculus for an operator $T$ should only depend on the values of the function on $\sigma(T)$ and so one would like a suitable algebra $AC(\sigma)$ where $\sigma$ is any nonempty compact subset of $\mC$. Over the years a number of papers addressed this problem (see for example, \cite{R2,BG1,BG2}) without providing a fully satisfactory theory.

For use in spectral theory, any proposed definition of a Banach algebra $AC(\sigma)$ of `absolutely continuous functions' on a compact set $\sigma \subseteq \mC$ should have the properties that:
\begin{enumerate}
 \item it should agree with the usual definition if $\sigma$ is an interval in $\mR$;
 \item $\AC(\sigma)$ should contain all sufficiently well-behaved functions;
 \item if $\alpha,\beta \in \mC$ with $\alpha \ne 0$, then the space $\AC(\alpha \sigma + \beta)$ should be isometrically isomorphic to $\AC(\sigma)$.
\end{enumerate}
The final condition is capturing the fact that the spectral decomposition and functional calculus for $\alpha T + \beta$ should essentially match that of $T$. 

Since none of the existing concepts of absolute continuity for functions defined on subsets of the plane satisfied these conditions, a new definition (which does satisfy properties 1, 2 and 3) was introduced in \cite{AD1} and these are the spaces $AC(\sigma)$ which we consider in this paper. We briefly outline the appropriate definitions in Section~\ref{prelims}.

An $AC(\sigma)$ operator is defined to be one which admits a bounded $AC(\sigma)$ functional calculus. This class includes not only all normal Hilbert space operators, but also the classes of scalar-type spectral operators, well-bounded operators and trigonometrically well-bounded operators acting on any Banach space \cite{AD1.5}. Compact $AC(\sigma)$ operators admit a spectral decomposition as a possibly conditionally convergent sum of the form (\ref{diag-form}) (see \cite{AD2}). Conversely, Theorem~5.1 of \cite{AD2} shows how to construct large families of compact $AC(\sigma)$ operators on a Banach space $X$ from conditional decompositions of $X$. In particular, $\alpha T$ is a compact $AC(\alpha \sigma(T))$ operator whenever $\alpha \in \mC$ and $T \in B(X)$ is compact and well-bounded.

For a normal operator $T$ on a Hilbert space, the $C^*$-algebra $C^*(T)$ generated by $T$ is isometrically isomorphic to $C(\sigma(T))$. The Banach--Stone and Gelfand--Kolmogorov theorems say that if $\Omega_1$ and $\Omega_2$ are compact Hausdorff spaces then the following conditions are equivalent.
\begin{enumerate}
 \item $\Omega_1$ and $\Omega_2$ are homeomorphic,
 \item $C(\Omega_1)$ and $C(\Omega_2)$ are linearly isometric as Banach spaces,
 \item $C(\Omega_1)$ and $C(\Omega_2)$ are isomorphic as algebras (or as $C^*$-algebras).
\end{enumerate}
This greatly limits the structure of the algebras $C^*(T)$ for compact normal operators.

The link between the Banach algebra
  \[ \mathcal{B}_T = \mathrm{cl}\{f(T) \st f \in AC(\sigma)\} \]
generated by an $AC(\sigma)$ operator and the function algebra $AC(\sigma)$ is less direct (even in the case that $\sigma = \sigma(T)$). Nonetheless, it is natural to ask about the extent to which a Banach--Stone type theorem might apply in this setting. In particular, we are predominently interested in determining conditions under which two such spaces $AC(\sigma_1)$ and $AC(\sigma_2)$ are isomorphic as Banach algebras.  

To be definite, if $\mathcal{A}$ and $\mathcal{B}$ are algebras, we shall say that $\mathcal{A}$ is isomorphic as an algebra to $\mathcal{B}$ if there exists an algebra isomorphism (that is, a linear and multiplicative bijection) $\Phi: \mathcal{A} \to \mathcal{B}$. If $\mathcal{A}$ and $\mathcal{B}$ are Banach algebras then we shall say that they are isomorphic as Banach algebras, and write $\mathcal{A} \simeq \mathcal{B}$, if there is an algebra isomorphism $\Phi: \mathcal{A} \to \mathcal{B}$ such that $\Phi$ and $\Phi^{-1}$ are continuous. Since this will generally be the context in this paper, unless otherwise specified, the term isomorphic should be taken mean isomorphic as Banach algebras.

In \cite{DL2} it was shown that if $AC(\sigma_1)$ and $AC(\sigma_2)$ are isomorphic as algebras, then any algebra isomorphism is necessarily bicontinuous and hence the spaces are in fact isomorphic as Banach algebras. However, unlike the case for the $C(\Omega)$ spaces, the algebra isomorphisms need not be isometric. 

One could also consider the question as to when $AC(\sigma_1)$ and $AC(\sigma_2)$ are isomorphic as Banach spaces. Although we will not pursue this here, we will see in Section~\ref{Op-alg} that it is easy to construct examples where two of these spaces are linearly isomorphic, but not isomorphic as algebras. 

One direction of the Gelfand--Kolmogorov theorem carries over to the current setting. If $AC(\sigma_1)$ and $AC(\sigma_2)$ are isomorphic as algebras then  the
sets $\sigma_1$ and $\sigma_2$ must be homeomorphic subsets of $\mC$ (see \cite[Theorem 2.6]{DL2}). 
The converse direction was shown to fail since the spaces of absolutely continuous functions on the closed disc and the closed square are not isomorphic. In a positive direction, the spaces for any two closed polygons are necessarily isomorphic Banach algebras. (Indeed, this result can be extended to a more general class of sets consisting of polygonal regions with polygonal holes.)

The aim of this paper is to examine the situation where $\sigma$ is the spectrum of a compact operator, and more particularly where $\sigma$ is a compact countable subset of the plane with a single limit point. All sets in this latter category are of course homeomorphic, but we shall show that one can get infinitely many non-isomorphic $AC(\sigma)$ spaces from such sets. On the other hand, if we require that these sets are subsets of $\mR$, then there are exactly two such spaces up to isomorphism. 

The $AC(\sigma)$ spaces are defined as subalgebras of spaces of functions of bounded variation, denoted by $BV(\sigma)$. 
For general sets $\sigma$ the isomorphisms between $BV(\sigma)$ spaces has been much less studied. Certainly isomorphisms of such spaces need not be associated with homeomorphisms of the domains of the function spaces (see Example~\ref{bv-no-hom} below). On the other hand, for the limited range of sets that we are working with here, $BV(\sigma)$ is not much bigger than $AC(\sigma)$ and many of the proofs can be adapted to give a result about these larger spaces.

\section{Preliminaries}\label{prelims}

In this section we shall briefly outline the definition of the spaces $\AC(\sigma)$ and $BV(\sigma)$.  Throughout, $\sigma$, $\sigma_1$ and $\sigma_2$ will denote nonempty compact subsets of the plane, which, for notational convenience we shall often identify as  $\mR^2$.
We shall work throughout with algebras of complex-valued functions.

Suppose that $f: \sigma \to \mC$. Let $S =
\plist{\vecx_0,\vecx_1,\dots,\vecx_n}$ be a finite ordered list of elements of $\sigma$, where, for the moment, we shall assume that $n \ge 1$. Let $\gamma_S$ denote the piecewise linear curve joining the points of $S$ in order.
Note that the elements of such a list do not need to be distinct.

The \textit{curve variation of $f$ on the set $S$} is defined to be
\begin{equation} \label{lbl:298}
    \cvar(f, S) =  \sum_{i=1}^{n} \abs{f(\vecx_{i}) - f(\vecx_{i-1})}.
\end{equation}
Unless $f$ is constant, this quantity can be made arbitrarily large by taking $S$ to consist of a repeating sequence of points on which $f$ differs. To deal with this we associate to each list $S$ a variation factor $\vf(S)$. Loosely speaking, this is the greatest number of times that $\gamma_S$ crosses any line in the plane.
To make this more precise we need the concept of a crossing segment.

\begin{defn}\label{crossing-defn}
Suppose that $\ell$ is a line in the plane. We say that $\ls[\vecx_i,\vecx_{i+1}]$, the line segment joining $\vecx_i$ to $\vecx_{i+1}$, is a \textit{crossing segment} of $S = \plist{\vecx_0,\vecx_1,\dots,\vecx_n}$ on $\ell$ if any one of the following holds:
\begin{enumerate}[(i)]
  \item $\vecx_i$ and $\vecx_{i+1}$ lie on (strictly) opposite sides of $\ell$.
  \item $i=0$ and $\vecx_i \in \ell$.
  \item $i > 0$, $\vecx_i \in \ell$ and $\vecx_{i-1} \not\in \ell$.
  \item $i=n-1$, $\vecx_i \not\in \ell$ and $\vecx_{i+1} \in \ell$.
\end{enumerate}
In this case we shall write $\ls[\vecx_i,\vecx_{i+1}] \in X(S,\ell)$.
\end{defn}

\begin{defn}\label{vf-defn}
Let $\vf(S,\ell)$ denote the number of crossing segments of $S$ on $\ell$. The \textit{variation factor} of $S$ is defined to be
 \[ \vf(S) = \max_{\ell} \vf(S,\ell). \]
\end{defn}

Clearly $1 \le \vf(S) \le n$. For completeness, in the case that
$S = \plist{\vecx_0}$ we set $\cvar(f, \plist{\vecx_0}) = 0$ and let $\vf(\plist{\vecx_0},\ell) = 1$ whenever $\vecx_0 \in \ell$.

\begin{defn}\label{2d-var}
The \textit{two-dimensional variation} of a function $f : \sigma
\rightarrow \mC$ is defined to be
\begin{equation} \label{lbl:994}
    \var(f, \sigma) = \sup_{S}
        \frac{ \cvar(f, S)}{\vf(S)},
\end{equation}
where the supremum is taken over all finite ordered lists of elements of $\sigma$.
\end{defn}

The \textit{variation norm} of such a function is
  \[ \normbv{f} = \norm{f}_\infty + \var(f,\sigma) \]
and the set of functions of bounded variation on $\sigma$ is
  \[ \BV(\sigma) = \{ f: \sigma \to \mC \st \normbv{f} < \infty\}. \]
The space $\BV(\sigma)$ is a Banach algebra under pointwise operations \cite[Theorem 3.8]{AD1}. If $\sigma = [a,b] \subseteq \mR$ then the above definition is equivalent to the more classical one.

Let $\Pol_2$ denote the space of polynomials in two real variables of the form $p(x,y) = \sum_{n,m} c_{nm} x^n y^m$, and let $\Pol_2(\sigma)$ denote the restrictions of elements on $\Pol_2$ to $\sigma$. The algebra $\Pol_2(\sigma)$ is always a subalgebra of $\BV(\sigma)$ \cite[Corollary 3.14]{AD1}.

\begin{defn}
The set of \textit{absolutely continuous} functions on $\sigma$, denoted $\AC(\sigma)$, is the closure of $\Pol_2(\sigma)$ in $\BV(\sigma)$.
\end{defn}

The set $\AC(\sigma)$ forms a closed subalgebra of $\BV(\sigma)$ and hence is a Banach algebra. Again, if $\sigma = [a,b]$ this definition reduces to the classical definition. 

More generally, we always have that $C^1(\sigma) \subseteq AC(\sigma) \subseteq C(\sigma)$, where one interprets $C^1(\sigma)$ as consisting of all functions for which there is a $C^1$ extension to an open neighbourhood of $\sigma$ (see \cite{DL1}).

\section{Locally piecewise affine maps}

It is essentially a consequence of the classical Banach--Stone Theorem that any algebra isomorphism between two $AC(\sigma)$ spaces must take the form of a  composition operator determined by a homeomorphism. (Note that this is not true for the $BV(\sigma)$ spaces.)

\begin{thm}[{\cite[Theorem~2.6]{DL2}}]\label{hom}
Suppose that $\sigma_1$ and $\sigma_2$ are nonempty compact subsets of the plane. If $\Phi: AC(\sigma_1) \to AC(\sigma_2)$ is an isomorphism then there exists a homeomorphism $h:\sigma_1 \to \sigma_2$ such $\Phi(f) = f \circ h^{-1}$ for all $f \in AC(\sigma_1)$.
\end{thm}

Not all homeomorphisms $h : \sigma_1 \to \sigma_2$ produce algebra isomorphisms, but a large class of suitable maps can be obtained by taking compositions of what are known as locally piecewise affine maps.

Let $\alpha: \mR^2 \to \mR^2$ be an invertible affine map, and let $C$ be a convex $n$-gon. Then $\alpha(C)$ is also a convex $n$-gon. Denote the sides of $C$ by $s_1,\dots,s_n$. Suppose that $\vecx_0 \in \interior{C}$. The point $\vecx_0$ determines a triangulation $T_1,\dots,T_n$ of $C$, where $T_j$ is the (closed) triangle with side $s_j$ and vertex $\vecx_0$. A point $\vecy_0 \in \interior{\alpha(C)}$ determines a similar triangularization $\hat{T}_1,\dots,\hat{T}_n$ of $\alpha(C)$, where the numbering is such that $\alpha(s_j) \subseteq \hat{T}_j$.

\begin{lem}\label{vertex-move} With the notation as above, there is a unique map $h: \mR^2 \to \mR^2$ such that
\begin{enumerate}
 \item $h(\vecx) = \alpha(\vecx)$ for $\vecx \not\in \interior{C}$,
 \item $h$ maps $T_j$ onto $\hat{T}_j$, for $1 \le j \le n$.
 \item $\alpha_j = h|T_j$ is affine, for $1 \le j \le n$.
 \item $h(\vecx_0) = \vecy_0$.
\end{enumerate}
\end{lem}

We shall say that $h$ is the \textit{locally piecewise affine} map determined by $(C,\alpha, \vecx_0,\vecy_0)$.

The important property of locally piecewise affine maps is that they preserve the isomorphism class of $AC(\sigma)$ spaces. (Explicit bounds on the norms of the isomorphisms are given in \cite{DL2}, but we shall not need these here. In any case, the known bounds are unlikely to be sharp.)

\begin{thm}\cite[Theorem 5.5]{AD2}\label{lpam-isom}
Suppose that $\sigma$ is a nonempty compact subset of the plane, and that $h$ is a locally piecewise affine map. Then $\BV(\sigma) \simeq \BV(h(\sigma))$ and $\AC(\sigma) \simeq \AC(h(\sigma))$.
\end{thm}

\begin{figure}[!ht]
\begin{center}
\begin{tikzpicture}[scale=0.6]
%
%
 \draw[thick,black] (-3,-1.5) -- (1.5,3) -- (-0.25,4.25) -- (-4.25,0.25) -- (-3,-1.5);
 \draw[blue,dashed]  (-3,-1.5) -- (-3,0) -- (1.5,3);
 \draw[blue,dashed]  (-0.25,4.25) -- (-3,0) -- (-4.25,0.25);
 \draw[fill,green!10] (2,-1) circle(1.3);
 \draw[black] (2,-1) circle(1.3);
 \draw (2,-1) node {$\sigma_1$};
 \draw[red,<->] (-4.5,0) -- (4,0) ;
 \draw[red,<->] (0,-2) -- (0,4.5);
 \draw[red] (-3,0) node[circle, draw, fill=black!50,inner sep=0pt, minimum width=4pt] {};
 \draw[black] (-3.2,0.4) node {$\vecx$};
 \draw[red] (0,3) node[circle, draw, fill=black!50,inner sep=0pt, minimum width=4pt] {};
 \draw[black] (0,3) node[right] {$\vecy$};
 \draw[black] (2,3) node[right] {$C$};

  \path[thick,->] (4.5,3) edge [bend left] (6.0,3);
  \draw[black] (5.25,3.3) node[above] {$h$};

 \draw[thick,black] (7,-1.5) -- (11.5,3) -- (9.75,4.25) -- (5.75,0.25) -- (7,-1.5);
 \draw[blue,dashed]  (7,-1.5) --(10,3) -- (11.5,3);
 \draw[blue,dashed]  (9.75,4.25) -- (10,3) -- (5.75,0.25);
 \draw[fill,green!10] (12,-1) circle(1.3);
 \draw[black] (12,-1) circle(1.3);
 \draw (12,-1) node {$\sigma_1$};
 \draw[red,<->] (5.5,0) -- (14,0) ;
 \draw[red,<->] (10,-1) -- (10,4.5);
 \draw[black] (12,3) node[right] {$C$};
 \draw[red] (7,0) node[circle, draw, fill=black!50,inner sep=0pt, minimum width=4pt] {};
 \draw[black] (7,0.1) node[above] {$\vecx$};
 \draw[red] (10,3) node[circle, draw, fill=black!50,inner sep=0pt, minimum width=4pt] {};
 \draw[black] (10,2.6) node[right] {$\vecy$};

\end{tikzpicture}
\caption{A locally piecewise affine map $h$ moving $\vecx$ to $\vecy$.}\label{lpam-pic}
\end{center}
\end{figure}

For most applications it suffices to restrict one's attention to locally piecewise affine maps where the map $\alpha$ is the identity. This allows you to move certain parts of $\sigma$ while leaving other parts fixed.  In particular, if $\sigma_1$ is a compact set and $\vecx $ and $\vecy $ are points in the complement of $\sigma_1$ which can be joined by a polygonal path which avoids $\sigma_1$, then $BV(\sigma_1 \cup \{\vecx \}) \simeq BV(\sigma_1 \cup \{\vecy \})$ and $AC(\sigma_1 \cup \{\vecx \}) \simeq AC(\sigma_1 \cup \{\vecy \})$ (see Figure~\ref{lpam-pic}). This would be sufficient to prove our main theorem in Section~\ref{S:C-sets}, but in the next section we shall prove a more general result which removes the requirement that there be a path from $\vecx $ to $\vecy $.

The following example shows that there are isomorphisms of $BV(\sigma)$ spaces which are not induced by homeomorphisms of the domains $\sigma_1$ ans $\sigma_2$.

\begin{ex}\label{bv-no-hom}
Let $\sigma_1 = \sigma_2 = \{0\} \cup \{\frac{1}{n}\}_{n=1}^\infty$. Define $h: \sigma_1 \to \sigma_2$ by
  \[ h(\vecx) = \begin{cases}
                 1, & \vecx = 0, \\
                 0, & \vecx = 1, \\
                 \vecx,  & \text{otherwise}.
                \end{cases} \]
and for  $f \in BV(\sigma_1)$ let $\Phi(f): \sigma_2 \to \mC$ be $\Phi(f) = f \circ h^{-1}$. A simple calculation shows that $\frac{1}{3} \var(f,\sigma_1) \le \var(\Phi(f),\sigma_2) \le 3 \var(f,\sigma_1)$ and so $\Phi$ is a Banach algebra isomorphism from $BV(\sigma_1)$ to $BV(\sigma_2)$. The map $h$ is of course not a homeomorphism. 
\end{ex}

On the other hand, as in the example, all isomorphisms of $BV(\sigma)$ spaces do come from composition with a bijection of the two domains.

\begin{thm}\label{BV-bij}
 Suppose that $\sigma_1$ and $\sigma_2$ are nonempty compact subsets of the plane. If $\Phi: BV(\sigma_1) \to BV(\sigma_2)$ is an isomorphism then there exists a bijection $h:\sigma_1 \to \sigma_2$ such $\Phi(f) = f \circ h^{-1}$ for all $f \in BV(\sigma_1)$.
\end{thm}

\begin{proof}
 Since $\Phi$ is an algebra isomorphism, it must map idempotents to idempotents. Note that for all $z \in \sigma_1$, the function $f_z = \chi_{\{z\}}$ lies in $BV(\sigma_1)$ and hence $g_z = \Phi(f_z)$ is an idempotent in $BV(\sigma_2)$. Since $\Phi$ is one-to-one, $g_z$ is not the zero function and hence the support of $g_z$ is a nonempty set $\tau \subseteq \sigma_2$. If $\tau$ is more than a singleton then we can choose $w \in \tau$ and write $g_z = \chi_{\{w\}} + \chi_{S\tau\setminus \{w\}}$ as a sum of two nonzero idempotents in $BV(\sigma_2)$. But then $f_z = \Phi^{-1}(\chi_{\{w\}}) + \Phi^{-1}(\chi_{S\setminus \{w\}})$ is the sum of two nonzero idempotent in $BV(\sigma_1)$ which is impossible. It follows that $g_z$ is the characteristic function of a singleton set and this clearly induces a map $h: \sigma_1 \to \sigma_2$ so that $\Phi(f_z) = \chi_{\{h(z)\}}$. Indeed, by considering $\Phi^{-1}$ it is clear that $h$ must be a bijection between the two sets.
\end{proof}

\section{Isolated points}

In general, calculating $\norm{f}_{BV(\sigma)}$, or indeed checking that a function $f$ is in $AC(\sigma)$ can be challenging. One way to simplify things is to break $\sigma$ into smaller pieces and then deal with the restrictions of $f$ to these pieces. If $\sigma_1$ is a compact subset of $\sigma$ and $f \in AC(\sigma)$ then it is easy to check that $f|\sigma_1 \in AC(\sigma_1)$ and $\norm{f|\sigma_1}_{BV(\sigma_1)} \le \norm{f}_{BV(\sigma)}$. However there are simple examples (see \cite[Example~3.3]{DL1}) where $\sigma = \sigma_1 \cup \sigma_2$, $f|\sigma_1 \in AC(\sigma_1)$, $f|\sigma_2 \in AC(\sigma_2)$, but $f \not\in BV(\sigma)$.

If one has disjoint sets $\sigma_1$ and $\sigma_2$, then the situation is rather better. Writing $\sigma =\sigma_1 \cup \sigma_2$  one essentially has \cite[Corollary~5.3]{DL1} that $AC(\sigma) = AC(\sigma_1) \oplus AC(\sigma_2)$.  To formally make sense of this one needs to identify $AC(\sigma_1)$ with the set $\{f \in AC(\sigma) \st \supp(f) \subseteq \sigma_1\}$. This requires that if one extends a function $g \in AC(\sigma_1)$ to all of $\sigma$ by making it zero on $\sigma_2$, then the extended function is absolutely continuous. While this is indeed always true, the constant $C_{\sigma_1,\sigma_2}$ such that $\norm{f}_{BV(\sigma)} \le C_{\sigma_1,\sigma_2} \norm{f|\sigma_1}_{BV(\sigma_1)}$ depends on the geometric configuration of the two sets, and is not bounded by any universal constant. 

For what we need later in the paper, we shall just need to consider the special case where $\sigma_2$ is an isolated singleton point. For the remainder of this section then assume that $\sigma_1$ is a nonempty compact subset of $\mC$, that $\vecz \not\in \sigma_1$ and that $\sigma = \sigma_1 \cup \{\vecz\}$. It is worth noting (using Proposition~4.4 of \cite{AD1} for example)  that $\chi_{\{\vecz\}}$ is always an element of $AC(\sigma)$.

For $f \in BV(\sigma)$ let 
  \[ \norm{f}_D = \norm{f}_{D(\sigma_1,\vecz)} =  \norm{f|\sigma_1}_{BV(\sigma_1)} + |f(\vecz)|. \]
(To prevent the notation from becoming too cumbersome we will usually just write $\norm{f}_{BV(\sigma_1)}$ rather than $\norm{f|\sigma_1}_{BV(\sigma_1)}$ unless there is some risk of confusion.)

\begin{prop}\label{D-norm}
The norm $\norm{\cdot}_D$ is equivalent to the usual norm $\norm{\cdot}_{BV(\sigma)}$ on $BV(\sigma)$.
\end{prop}

\begin{proof}
We first remark that it is clear that $\norm{\cdot}_D$ is a norm on $BV(\sigma)$. Also, noting the above remarks, 
  $ \norm{f}_D \le 2 \norm{f}_{BV(\sigma)} $
so we just need to find a suitable lower bound for $\norm{f}_D$.

Suppose then that $f \in BV(\sigma)$. Let $S = [\vecx_0,\vecx_1,\dots,\vecx_n]$ be an ordered list of points in $\sigma$ and let $S' = [\vecy_0,\dots,\vecy_m]$   be the list $S$ with those points equal to $\vecz$ omitted. Our aim is to compare $\cvar(f,S)$ with $\cvar(f,S')$. 
In calculating $\cvar(f,S)$  we may assume that no two consecutive points in this list are both equal to $\vecz$,  and that $S'$ is nonempty. Let $N$ be the number of times that the point $\vecz$ occurs in the list $S$.

Now if $\vecx_k = \vecz$ for some $0 < k < n$ then
  \begin{multline*}
   |f(\vecx_k) - f(\vecx_{k-1})| + |f(\vecx_{k+1}) - f(\vecx_k)| \\
   \le 2 \norm{f|\sigma_1}_\infty + 2 |f(\vecz)| \\
   \le |f(\vecx_{k+1}) - f(\vecx_{k-1})| + 2 \norm{f|\sigma_1}_\infty + 2 |f(\vecz)|
  \end{multline*}
If $\vecx_0 = \vecz$ then $|f(\vecx_1)-f(\vecx_0)| \le \norm{f|\sigma_1}_\infty + |f(\vecz)|$ and a similar estimate applies if $\vecx_n = \vecz$.
Putting these together shows that\goodbreak
  \begin{align*}
   \cvar(f,S) &= \sum_{k=1}^n |f(\vecx_k) - f(\vecx_{k-1})| \\
              &\le \sum_{k=1}^m |f(\vecy_k) - f(\vecy_{k-1})| + 2N (\norm{f|\sigma_1}_\infty+|f(\vecz)|).
  \end{align*}
Let $\ell$ be any line through $\vecz$ which doesn't intersect any other points of $S$. Checking Definition~\ref{crossing-defn}, one sees that we get a crossing segment of $S$ on $\ell$ for each time that $\vecx_k = \vecz$ and so $\vf(S) \ge \vf(S,\ell) \ge N$.
By \cite[Proposition~3.5]{DL1} we also have that $\vf(S) \ge \vf(S')$.
Thus
  \begin{align*}
   \frac{\cvar(f,S)}{\vf(S)} 
     &\le \frac{\cvar(f,S') + 2N (\norm{f|\sigma_1}_\infty+|f(\vecz)|)}{\vf(S)} \\
     &\le \frac{\cvar(f,S')}{\vf(S')} + \frac{2N (\norm{f|\sigma_1}_\infty+|f(\vecz)|)}{N} \\
     &\le \var(f,\sigma_1) + 2 (\norm{f|\sigma_1}_\infty+|f(\vecz)|) \\
     &\le 2 \norm{f}_D.
   \end{align*}
Taking the supremum over all lists $S$ then shows that $\var(f,\sigma) \le 2 \norm{f}_D$ and hence that 
  \[ \norm{f}_{BV(\sigma)} = \norm{f}_\infty+  \var(f,\sigma) \le 3 \norm{f}_D \] 
which completes the proof.
\end{proof}

The constants obtained in the proof of Proposition~\ref{D-norm} are in fact sharp. Suppose that $\sigma_1 = \{-1,1\}$, $\vecz = 0$ and $\sigma = \sigma_1 \cup \{z\}$. Then $\norm{\chi_{\{0\}}}_D = 1$ while $\norm{\chi_{\{0\}}}_{BV(\sigma)} = 3$. On the other hand, if $f$ is the constant function $1$, then $\norm{f}_D = 2$ while $\norm{f}_{BV(\sigma)} =1$.

\begin{prop}\label{pres-ac}
$f \in AC(\sigma)$ if and only if $f|\sigma_1 \in AC(\sigma_1)$.
\end{prop}

\begin{proof} Rather than use the heavy machinery of \cite[Section~5]{DL1}, we give a more direct proof using the definition of absolute continuity. As noted above, one just needs to show that if $f|\sigma_1 \in AC(\sigma_1)$, then $f \in AC(\sigma)$. Suppose then that $f|\sigma_1 \in AC(\sigma_1)$. Given $\epsilon > 0$, there exists a polynomial $p \in \Pol_2$ such that $\norm{f-p}_{BV(\sigma_1)} < \epsilon/3$.
Define $g: \sigma \to \mC$ by $g = p + (f(\vecz)-p(\vecz)) \chi_{\{z\}}$. Since $\chi_{\{z\}} \in AC(\sigma)$, we have that $g \in AC(\sigma)$ and $\norm{f-g}_{BV(\sigma)} \le 3 \norm{f-g}_D = 3 \norm{f-p}_{BV(\sigma_1)} < \epsilon$. Since $AC(\sigma)$ is closed, this shows that $f \in AC(\sigma)$.
\end{proof}

\begin{cor}\label{isol-points}
Suppose that $\sigma_1$ is a nonempty compact subset of $\mC$ and that $\vecx$ and $\vecy$ are points in the complement of $\sigma_1$. Then 
$BV(\sigma_1 \cup \{\vecx\}) \simeq BV(\sigma_1 \cup \{\vecy\})$
$AC(\sigma_1 \cup \{\vecx\}) \simeq AC(\sigma_1 \cup \{\vecy\})$.
\end{cor}

\begin{proof}  Let $h: \sigma_1 \cup \{\vecx\} \to \sigma_1 \cup \{\vecy\}$ be the natural homeomorphism which is the identity on $\sigma_1$ and which maps $\vecx$ to $\vecy$ and for $f \in BV(\sigma_1 \cup \{\vecx\})$ let  $\Phi(f) = f \circ h^{-1}$. Then $\Phi$ is an algebra isomorphism of $BV(\sigma_1 \cup \{\vecx\})$ onto $BV(\sigma_1 \cup \{\vecy\})$ which is isometric under the norms $\norm{\cdot}_{D(\sigma_1,\vecx)}$ and $\norm{\cdot}_{D(\sigma_1,\vecy)}$, and hence it is certainly bicontinuous under the respective $BV$ norms.

It follows immediately from Proposition~\ref{pres-ac} that $\Phi$ preserves absolute continuity as well. 
\end{proof}

More generally of course, this result says that one can move any finite number of isolated points around the complex plane without altering the isomorphism class of these spaces.

\section{$C$-sets}\label{S:C-sets}

The spectrum of a compact operator is either finite or else a countable set with limit point 0. If $\sigma$ has $n$ elements then $AC(\sigma)$ is an $n$-dimensional algebra and consequently for finite sets, one has a trivial Banach--Stone type theorem: $AC(\sigma_1) \simeq AC(\sigma_2)$ if and only if $\sigma_1$ and $\sigma_2$ have the same number of elements. (Of course the same result is also true for the $BV(\sigma)$ spaces.)

The case where $\sigma$ is a countable set is more complicated however.

\begin{defn}
We shall say that a subset $\sigma \subseteq \mC$ is a \textit{$C$-set} if it is a countably infinite compact set with unique limit point $0$. If further $\sigma \subseteq \mR$ we shall say that $\sigma$ is a \textit{real $C$-set}.
\end{defn}

Any two $C$-sets are homeomorphic, but as we shall see, they can produce an infinite number of non-isomorphic spaces of absolutely continuous functions. In most of what follows, it is not particularly important that the limit point of the set is $0$ since one can apply a simple translation of the domain $\sigma$ to achieve this and any such translation induces an isometric isomorphism of the corresponding function spaces.

The easiest $C$-sets to deal with are what were called spoke sets in \cite{AD2}, that is, sets which are contained in a finite number of rays emanating from the origin. To state our main theorem, we shall need a slight variant of this concept. For $\theta \in [0,2\pi)$ let $R_\theta$ denote the ray $\{t e^{i\theta} \st t \ge 0\}$.

\begin{defn}\label{k-ray}
Suppose that $k$ is a positive integer. We shall say that a $C$-set $\sigma$ is a \textit{$k$-ray set} if there are $k$ distinct rays $R_{\theta_1},\dots,R_{\theta_k}$ such that 
\begin{enumerate}
 \item $\sigma_j := \sigma \cap R_{\theta_j}$ is infinite for each $j$,
 \item $\sigma_0 := \sigma \setminus (\sigma_1 \cup \dots \cup \sigma_k)$ is finite.
\end{enumerate}
If $\sigma_0$ is empty then we shall say that $\sigma$ is a \textit{strict $k$-ray set}.
\end{defn}

Although in general the calculation of norms in $BV(\sigma)$ can be difficult, if $\sigma$ is a strict $k$-ray set, then we can pass to a much more tractable equivalent norm, called the spoke norm in \cite{AD2}.

\begin{defn}
Suppose that $\sigma$ is a strict $k$-ray set. The \textit{$k$-spoke norm} on $BV(\sigma)$ is (using the notation of Definition~\ref{k-ray})
  \[ \norm{f}_{Sp(k)} = |f(0)| + \sum_{j=1}^k \norm{f-f(0)}_{BV(\sigma_j)}. \]
\end{defn}

Since each of the subsets $\sigma_j$ is contained in a straight line, the calculation of the variation over these sets is straightforward. If we write $\sigma_j = \{0\} \cup \{\lambda_{j,i}\}_{i=1}^\infty$ with $|\lambda_{j,1}| > |\lambda_{j,2}| > \dots$ then
  \begin{equation*}
  \norm{f-f(0)}_{BV(\sigma_j)} = \sup_i |f(\lambda_{j,i}) - f(0)| + \sum_{i=1}^\infty |f(\lambda_{j,i}) - f(\lambda_{j,i+1})|. 
  \end{equation*}

\begin{prop}[{\cite[Proposition 4.3]{AD2}}]\label{sp-norm-eq} Suppose that $\sigma$ is a strict $k$-ray set. Then for all $f \in BV(\sigma)$, 
   \[ \frac{1}{2k+1} \norm{f}_{Sp(k)} \le \norm{f}_{BV(\sigma)} \le 3 \norm{f}_{Sp(k)}. \]
\end{prop}

One property which significantly simplifies the analysis for such spaces is that for a $k$-ray set $\sigma$, one always has that $AC(\sigma) = BV(\sigma) \cap C(\sigma)$. In particular a function of bounded variation on such a set $\sigma$ is absolutely continuous if and only if it is continuous at the origin. 

\begin{prop} \label{AC-kset} If $\sigma$ is a $k$-ray set then $AC(\sigma) = BV(\sigma) \cap C(\sigma)$.
\end{prop}

\begin{proof}
Since $AC(\sigma)$ is always a subset of $BV(\sigma) \cap C(\sigma)$ we just need to prove the reverse inclusion. 

Suppose first that $\sigma$ is a strict $k$-ray set, and suppose that $f \in BV(\sigma) \cap C(\sigma)$. For $n = 1,2,\dots$, define $g_n:\sigma \to \mC$ by
  \[ g_n(\vecz) = \begin{cases}
               f(\vecz), & \text{if $|\vecz| \ge \frac{1}{n}$,} \\
               f(0),    & \text{if  $|\vecz| < \frac{1}{n}$}
              \end{cases} 
              \quad = f(0) + \sum_{|\vecz| \ge 1/n} (f(\vecz) - f(0)) \chi_{\{\vecz\}}.
              \]
Since $\chi_{\{\vecz\}}\in AC(\sigma)$ for all nonzero points in $\sigma$, $g_n \in AC(\sigma)$. Now 
  \begin{equation}\label{gn-approx}
   \norm{f - g_n}_{Sp(k)} = \sum_{j=1}^k \norm{f - g_n}_{BV(\sigma_j)}.
  \end{equation}
Fix $j$ and label the elements of $\sigma_j$ as above. Then, for all $n$ there exists an index $I_{j,n}$ such that $| \lambda_{j,i} | < \frac{1}{n}$ if and only if $i \ge I_{j,n}$. Thus
  \begin{multline*}
   \norm{f - g_n}_{BV(\sigma_j)} \\
   = \sup_{i \ge I_{j,n}} | f(\lambda_{j,i}) - f(0)| 
             + \sum_{i \ge I_{j,n}} |f(\lambda_{j,i}) - f(\lambda_{j,i+1})|
             + |f(\lambda_{j,I_{j,n}} - f(0)| .
  \end{multline*}
The first and last of these terms converge to zero since $f \in C(\sigma)$. The middle term also converges to zero since it is the tail of a convergent sum.

Since we can make each of the $k$ terms in (\ref{gn-approx}) as small as we like, $\norm{f - g_n}_{Sp(k)} \to 0$ and hence $g_n \to f$ in $BV(\sigma)$. Thus $f \in AC(\sigma)$.

Suppose finally that $\sigma$ is not a strict $k$-ray set, that is  that $\sigma_0 \ne \emptyset$. Let $\sigma' = \sigma\setminus \sigma_0$.  If $f \in BV(\sigma) \cap C(\sigma)$, then $f|\sigma' \in BV(\sigma') \cap C(\sigma')$. By the above, $f|\sigma' \in AC(\sigma')$. Repeated use of Proposition~\ref{pres-ac} then shows that $f \in AC(\sigma)$. 
\end{proof}

It would be interesting to know whether Theorem~\ref{AC-kset} holds for more general $C$-sets.

\begin{cor}
Suppose that $\sigma$ is a strict $k$-ray set and that $f: \sigma \to \mC$. For $j = 1,\dots,k$, let $f_j$ denote the restriction of $f$ to $\sigma_j$. Then $f \in AC(\sigma)$ if and only if $f_j \in AC(\sigma_j)$ for all $j$.
\end{cor}

\begin{proof} By Lemma~4.5 of \cite{AD1} if $f \in AC(\sigma)$ then the restriction of $f$ to any compact subset is also absolutely continuous. If each $f_j \in AC(\sigma_j)$, then certainly $f \in C(\sigma)$. Furthermore $\norm{f}_{Sp(k)}$ is finite and hence $f \in BV(\sigma)$. Thus, by Proposition~\ref{AC-kset}, $f \in AC(\sigma)$. 
\end{proof}

\begin{thm}\label{Thm:k-sets}
Suppose that $\sigma$ is a $k$-ray set and that $\tau$ is an $\ell$-ray set. Then $AC(\sigma) \simeq AC(\tau)$ if and only if $k = \ell$.
\end{thm}

\begin{proof} 
Write $\sigma = \cup_{j=0}^k \sigma_j$ and $\tau = \cup_{j=0}^\ell \tau_j$ as in Definition~\ref{k-ray}.
It follows from Corollary~\ref{isol-points} that by moving the finite number of points in $\sigma_0$ onto one of the rays containing a set $\sigma_j$, that $AC(\sigma)$ is isomorphic to $AC(\sigma')$ for some strict $k$-ray set. To prove the theorem then, it suffices therefore to assume that 
$\sigma$ and $\tau$ are strict $k$ and $\ell$-ray sets.

Suppose first that $k > \ell$ and that there is a Banach algebra isomorphism $\Phi$ from $AC(\sigma)$ to $AC(\tau)$. By Theorem~\ref{hom}, $\Phi(f) = f \circ h^{-1}$ for some homeomorphism $h: \sigma \to \tau$.

By the pigeonhole principle there exists $L \in \{1,\dots,\ell\}$ such that $h(\sigma_j) \cap \tau_L$ is infinite for (at least) two distinct sets values of $j$. Without loss of generality we will assume that this is true for $j = 1$ and $j = 2$. Indeed, since rotations produce isometric isomorphisms of these spaces, we may also assume that $\tau_L \subset [0,\infty)$. Let $\sigma_j = \{0\} \cup \{\lambda_{j,i}\}_{i=1}^\infty$, where the points are labelled so that $|\lambda_{j,1}| > |\lambda_{j,2}| > \dots$. There must then be two increasing sequences $i_1 < i_2 < \dots$ and $k_1 < k_2 < \dots$ such that 
  \[ h(\lambda_{1,i_1})  > h(\lambda_{2,k_1}) > h(\lambda_{1,i_2}) > h(\lambda_{2,k_2}) > \dots .\]
For $n = 1,2,\dots$ define $f_n \in AC(\sigma)$ by
  \[ f_n(\vecz) = \begin{cases}
                   1, & \vecz \in \{\lambda_{1,1},\dots,\lambda_{1,n}\} \\
                   0, & \text{otherwise.}
                  \end{cases}
   \] 
Then $\norm{f_n}_{Sp(k)} = 2$ for all $n$, but $\norm{\Phi(f_n)}_{Sp(\ell)} \ge 2n$. Using Proposition~\ref{sp-norm-eq}, this means that $\Phi$ must be unbounded which is impossible. Hence no such isomorphism can exist.

Finally, suppose that $k = \ell$. For each $j = 1,2,\dots,k$ order the elements of $\sigma_j$ and $\tau_j$ by decreasing modulus and let $h_j$ be the unique homeomorphism from $\sigma_j$ to $\tau_j$ which preserves this ordering. Let $h$ be the homeomorphism whose restriction to each $\sigma_j$ is $h_j$ and let $\Phi(f) = f \circ h^{-1}$. Then $\Phi$ is an isometric isomorphism from $(BV(\sigma),\norm{\cdot}_{Sp(k)})$ to $(BV(\tau),\norm{\cdot}_{Sp(k)})$, and hence is a Banach algebra isomorphism between these spaces under their usual $BV$ norms. Since $\Phi$ is also an isomorphism from $C(\sigma)$ to $C(\tau)$, the result now follows from Proposition~\ref{AC-kset}.
\end{proof}

\begin{cor}\label{inf-many-noniso}
There are infinitely many mutually non-isomorphic $AC(\sigma)$ spaces with $\sigma$ a $C$-set.
\end{cor}

Clearly any real $C$-set is either a $1$-ray set, or a $2$-ray set.

\begin{cor}
There are exactly two isomorphism classes of $AC(\sigma)$ spaces with $\sigma$ a real $C$-set.
\end{cor}

We should point out at this point that Theorem~\ref{Thm:k-sets} is far from a characterization of the sets $\tau$ for which $AC(\tau)$ is isomorphic to $AC(\sigma)$ where $\sigma$ is some $k$-ray set.

\begin{ex}
Let $\tau = \{0\} \cup \bigl\{ \frac{1}{j} + \frac{i}{j^2} \bigr\}_{j=1}^\infty$ and let $\sigma = \{0\} \cup \bigl\{ \frac{1}{j} \bigr\}_{j=1}^\infty$. Clearly $\tau$ is not a $k$-ray set for any $k$. 

For $f  \in BV(\sigma)$ let $\Phi(f)(t+it^2) = f(t)$, $t \in \sigma$. It follows from  \cite[Lemma~3.12]{AD1} that $\norm{\Phi(f)}_{BV(\tau)} \le \norm{f}_{BV(\sigma)}$. For the other direction, suppose that $\lambda_0 \le \lambda_1 \le\dots \le \lambda_n$ are points in $\sigma$ and let $S = [\lambda_0+i \lambda_0^2,\dots,\lambda_n+i \lambda_n^2]$ be the corresponding list of points in $\tau$.  It is easy to see that $\vf(S)$ is 2 if $n > 1$ (and is 1 if $n = 1$). Then
  \begin{align*}
   \sum_{j=1}^n |f(\lambda_j)) - f(\lambda_{j-1})| 
     & = \sum_{j=1}^n |\Phi(f)(\lambda_j+i\lambda_j^2)) - \Phi(f)(\lambda_{j-1}+i\lambda_j^2)| \\
     & \le 2 \frac{\cvar(\Phi(f),S)}{\vf(S)} \\
     & \le 2 \var(\Phi(f),\tau).
  \end{align*}
Since the variation of $f$ is given by the supremum of such sums over all such ordered subsets of $\sigma$, $\var(f,\sigma) \le 2 \var(\Phi(f),\tau)$ and hence $\norm{f}_{BV(\sigma)} \le 2 \norm{\Phi(f)}_{BV(\tau)}$. This shows that $BV(\sigma) \simeq BV(\tau)$. 

Proposition~4.4 of \cite{AD1} ensures that if $f \in AC(\sigma)$ then $\Phi(f) \in AC(\tau)$. Conversely, if $g = \Phi(f) \in AC(\tau)$ then certainly $g \in C(\tau)$ and consequently $f \in C(\sigma)$. By the previous paragraph $f \in BV(\sigma)$ too and hence, by Proposition~\ref{AC-kset}, $f \in AC(\sigma)$. Thus $AC(\sigma) \simeq AC(\tau)$.
\end{ex}

\begin{ex}
Let $\sigma = \{0\} \cup \bigl\{ \frac{e^{i/m}}{n} \st n,m \in \mZ^+\bigr\} \cup \bigl\{ \frac{1}{n} \st n \in \mZ^+\bigr\}$ (where $\mZ^+$ denotes the set of positive integers) and let $\tau$ be an $\ell$-ray set. Repeating the proof of Theorem~\ref{Thm:k-sets}, one sees that there can be no Banach algebra isomorphism from $AC(\sigma)$ to $AC(\tau)$ so even among $C$-sets there are more isomorphism classes than those captured by Theorem~\ref{Thm:k-sets}. 
\end{ex}

The corresponding result for the $BV(\sigma)$ spaces also holds.

\begin{cor}
Suppose that $\sigma$ is a $k$-ray set and that $\tau$ is an $\ell$-ray set. Then $BV(\sigma) \simeq BV(\tau)$ if and only if $k = \ell$.
\end{cor}

\begin{proof}
The proof is more or less identical to that of Theorem~\ref{Thm:k-sets}. In showing that if $k \ne\ell$ then $AC(\sigma) \not\simeq AC(\tau)$ we used the fact that any isomorphism between these spaces is of the form $\Phi(f) = f \circ h^{-1}$. In showing that such a map cannot be bounded, the continuity of $h$ was not used, only the fact that $h$ must be a bijection, and so one may use Theorem~\ref{BV-bij} in place of Theorem~\ref{hom} in this case.

The fact that if $k = \ell$ then $BV(\sigma) \simeq BV(\tau)$ is already noted in the above proof.
\end{proof}

\section{Operator algebras}\label{Op-alg}

If $\sigma = \{0\} \cup \bigl\{\frac{1}{n}\bigr\}_{n=1}^\infty$, the map $\Psi: AC(\sigma) \to \ell^1$, 
  \[ \Psi(f) = (f(1),f(\tfrac{1}{2}) - f(1), f(\tfrac{1}{3}) - f(\tfrac{1}{2}), \dots) \]
is a Banach space isomorphism. Indeed it is not hard to see that Proposition~\ref{sp-norm-eq} implies that if $\sigma$ is a strict $k$-ray set, then, as Banach spaces, $AC(\sigma)$ is isomorphic to $\oplus_{j=1}^k \ell^1$ which in turn is isomorphic to $\ell^1$, and consequently all such $AC(\sigma)$ spaces are Banach space isomorphic.

Given any nonempty compact set $\sigma \subseteq \mC$, the operator $Tg(z) = zg(z)$ acting on $AC(\sigma)$ is an $AC(\sigma)$ operator. Indeed the functional calculus for $T$ is given by $f(T)g = fg$ for $f \in AC(\sigma)$ from which one can deduce that $\norm{f(T)} = \norm{f}_{BV(\sigma)}$, and therefore the Banach algebra generated by the functional calculus for $T$ is isomorphic to $AC(\sigma)$. Proposition~6.1 of \cite{AD2} shows that if $\sigma$ is a $C$-set then any such operator $T$ is a compact $AC(\sigma)$ operator.

Combining these observations, together with Corollary~\ref{inf-many-noniso},  shows that on $\ell^1$ there are infinitely many nonisomorphic Banach subalgebras of $B(\ell^1)$ which are generated by (non-finite rank) compact $AC(\sigma)$ operators on $\ell^1$, so things are rather different to the situation for compact normal operators on $\ell^2$.

%
%
\bibliographystyle{amsalpha}

\end{document}